\documentclass{daj}

\dajAUTHORdetails{%
  title = {The magnitude function determines generic finite metric spaces}, 
  author = {Jun O'Hara},
  plaintextauthor = {Jun O'Hara},
  plaintexttitle = {Magnitude function determines generic finite metric spaces}, 
  runningtitle = {Magnitude function determines}, 
  runningauthor = {Jun O'Hara},
  copyrightauthor = {Jun O'Hara},
  keywords = {magnitude, finite metric space},
}   

\dajEDITORdetails{%
   year={2025},
   number={13},
   received={20 January 2024},   
   revised={11 May 2025},    
   published={8 September 2025},  
   doi={10.19086/da.143788},       
}   

\usepackage{amssymb,amsfonts,amsmath,latexsym,color,amscd} 
\usepackage[all,cmtip]{xy}
\usepackage{url}
\usepackage{mathrsfs}

\newtheorem{theorem}{Theorem}[section]
\newtheorem{lemma}[theorem]{Lemma}
\newtheorem{corollary}[theorem]{Corollary}
\newtheorem{proposition}[theorem]{Proposition}

\newtheorem{definition}[theorem]{Definition}

\newenvironment{proof}{{\par\addvspace{0.1cm}\noindent \bf Proof. }}{\hfill$\Box$\par\medskip}

\newtheorem{remark}[theorem]{Remark}

\def\a{\alpha}
\def\RR{\mathbb{R}}
\def\QQ{\mathbb{Q}}
\def\NN{\mathbb N}

\def\de{\delta}
\def\la{\lambda}
\def\si{\sigma}
\def\ga{\gamma}

\def\sotp{{\,\mbox{\rm \footnotesize s.o.3-p.}}}
\def\scsotp{{\,\mbox{\rm \scriptsize s.o.3-p.}}}
\def\otp{{\mbox{\rm \footnotesize o.3-p.}}}
\def\scotp{{\mbox{\rm \scriptsize o.3-p.}}}
\def\otwop{{\mbox{\rm \footnotesize o.2-p.}}}
\def\scotwop{{\mbox{\rm \scriptsize o.2-p.}}}

\begin{document}

\begin{frontmatter}[classification=text]

\title{Magnitude function determines generic finite metric spaces} 

\author[jo]{Jun O'Hara\thanks{Supported by JSPS KAKENHI Grant Number 19K03462.}}
\begin{abstract}
We give sufficient conditions for a finite metric space to be determined by the magnitude function. 
In particular, a generic finite metric space such that the distances between the points are rationally independent is determined by the asymptotic behavior of the magnitude function. 
\end{abstract}
\end{frontmatter}



\section{Introduction}

\phantom{a}
\vspace{-0.5cm}

The magnitude is a numerical invariant for a metric space introduced by Leinster \cite{L13} using the theory of enriched categories. It gives a function, which we call the magnitude function, on $(0,\infty)$ as the magnitude of the space similarly expanded by $t$ times. 
The magnitude function has information of the distance and scale of the space. 
Since its introduction, it has been intensively studied along with the magnitude (co)homology derived from it (\cite{HW}, \cite{LS}), involving various fields of mathematics, not only algebraic but also analytic. 

Let $(X,d)$ with $X=\{P_1,\dots,P_n\}$ be a finite metric space. We will call each pair $(P_i,P_j)$ $(i\ne j)$ an edge and the distance between them $d(P_i,P_j)$ the {\em edge length} and denote it by $d_{ij}$. 
Let $tX$ $(t>0)$ denote the metric space $(X,td)$. 
The {\em similarity matrix} $Z_X(t)$ is given by $Z_X(t)=\left(\exp(-t\, d_{ij})\right)_{i,j}$. 
In this paper we only use the cases when $Z_X(t)$ is invertible. 
Then the {\em magnitude function}, denoted by $|tX|$ or $M_X(t)$, is given by the sum of all the entries of $Z_X(t)^{-1}$. 
It is known that $M_X(t)$ is an increasing function of $t$ for $t\gg0$ and that $\lim_{t\to\infty}M_X(t)=\#X$ (\cite{L13} Proposition 2.2.6). 
A space $X$ is said to have {\em one-point property} if $\lim_{t\to0^+}M_X(t)=1$. 
Any compact subset of Euclidean space has one-point property (\cite{LM23} Theorem 3.1). 

An alternative expression of the magnitude can be obtained by putting $q=e^{-t}$ (\cite{L19}). The similarity matrix is then given by $z_X(q)=\left(q^{d_{ij}}\right)_{i,j}$. 
The determinant and the sum of all the cofactors of $z_X(q)$ are both ``{\em generalized polynomials}'' that allow non-integer exponents. 
Since the determinant has the constant term $1$, it is invertible in the field of generalized rational functions. The sum of all the entries of $z_X(q)^{-1}$ is called the {\em formal magnitude}, and denoted by $m_X(q)$. 
In this article we assume that $m_X(q)$ is expanded as a ``{\em generalized formal power series}'' that allows non-integer exponents. 

One of the basic questions would be to what extent a space can be identified by the magnitude. 
Gimperlein, Goffeng and Louca showed that in the case of smooth manifolds with boundaries, information such as the volumes of the manifolds and their boundaries, and the integrals of (covariant derivatives of) curvatures can be obtained from the asymptotic expansion of the magnitude function at large scale (\cite{GGL} Theorem 2.1). It follows that balls can be identified by the magnitude function. 

In this article we study the problem for finite metric spaces. 
In the case of graphs, there exist examples where the magnitudes are the same but the graphs are not isometric. For example, Leinster gave the example in Figure \ref{two_graphs} (\cite{L13} Example 2.3.5)\footnote{Even if we change the lengths of the three edges from $1$ to $a, b,c$ in the right and left graphs, the magnitude functions of the two graphs are the same. It follows that the left graphs with edge lengths $a,b,c$ in order and $b,a,c$ in order have the same magnitude although they are not isometric if $a\ne b$. This gives an example of a pair of finite subspaces of Euclidean space that are not isometric but have the same magnitude. }. 
\begin{figure}[htbp]
\begin{center}
\includegraphics[width=.6\linewidth]{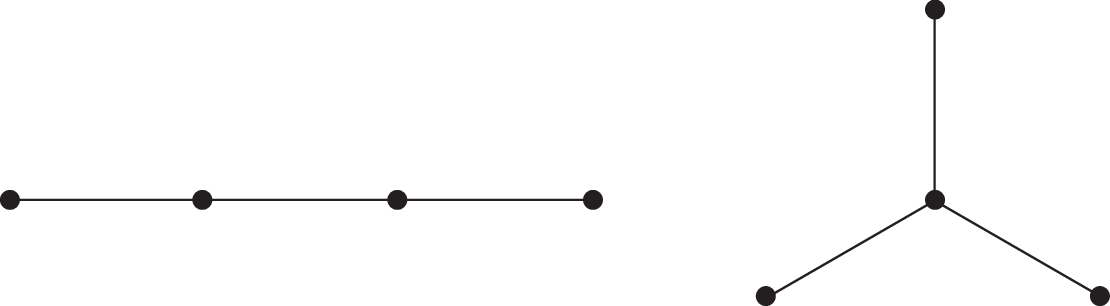}
\caption{Both graphs have the same magnitude $\displaystyle \frac{4-2q}{1+q}$ $(q=e^{-t})$.}
\label{two_graphs}
\end{center}
\end{figure}
There are many such examples obtained by applying Whitney twist (\cite{L19} Section 5). 

On the other hand, if we restrict ourselves to finite subsets of Euclidean spaces, we can expect magnitude to be a powerful tool for distinguishing spaces, since there are examples of spaces that cannot be distinguished by the discrete version of natural concepts in integral geometry but can be distinguished by magnitude. 
In fact, numerical experiments shows that the magnitude can distinguish 30 possible positions of four vertices of tetrahedra with the same set of edge lengths $\{7,8,9,10,11,12\}$. 
Here, the set of edge lengths with multiplicity can be considered as a discrete analogue of the distribution of interpoint distances, which is one of the basic notions in integral geometry. 
Exchanging the lengths of two edges could change the inverse of the similarity matrix $Z_X(t)$ and thus the magnitude. Conceptually speaking, we can say that the magnitude includes not only the edge length information, but also some combinatorial information, which helps us to identify the space.

It should be noted that the distribution of interpoint distances is also useful for identification of spaces. 
Through the Mellin transform, it yields {\sl Brylinski's beta function} (\cite{Bry}), which gives various geometric information such as the volumes of the manifold and its boundary, and integral of curvatures (\cite{FV} Theorem 4.1, \cite{OS} Proposition 4.7). 
Thus the balls and, under certain conditions, the circles and $2$-spheres can be identified by Brylinski's beta function (\cite{O19} Theorem 3.4). 
For a convex body (i.e. a compact convex subset of Euclidean space with non-empty interior), the interpoint distance distribution is equivalent to another basic notion in integral geometry, the distribution of chord lengths of the intersection of the convex body and random lines. 
Blaschke asked (\cite{B} p.51) if the planar domain is determined by the chord length distribution. 
The answer is no since there is a counterexample (\cite{MC} Figure 1), although Waksman claimed that a sufficiently asymmetric convex polygons is characterized by this distribution (\cite{W} Theorem 5.6). 

In this paper we give sufficient conditions for a finite metric space $X$ to be determined by the magnitude function. 
If the edge lengths are rationally independent, a finite metric space is determined by the asymptotic behavior of the magnitude function at large scale $t\to +\infty$. 
When $\#X=3$, $X$ is determined by the asymptotic behavior of the magnitude function at small scale $t\to 0^+$ without any condition. 
If we use Roff-Yoshinaga's parametrization of finite metric spaces (\cite{RY} Section 2), what is excluded by our genericity condition is codimension one and hence measure zero. 

%
%

\section{Main Theorem}

\begin{definition} \rm 
Let $X$ be a metric space which consists of $n$ points and let $\ell_1,\dots, \ell_N$ be the edge lengths, where $N={n\choose2}$. 
A map 
\[A\colon \left\{\{i,j\}\,|\,1\le i<j\le n \right\}\to\{\ell_1,\dots,\ell_N\}\]
is called {\em combinatorial data} (of edge lengths) when there exists a labeling of points in $X$ such that $A(\{i,j\})=d(P_i,P_j)$ for any $i,j$ $(i<j)$. 

We write $d_{i,j}=d(P_i,P_j)$ in what follows. 
\end{definition}

A {\em multiset} is a set with multiplicity. We will use the symbol $[\>]$ for multisets. For example, although $\{a,a,b\}$ and $\{a,b\}$ are same as a set, $[a,a,b]$ and $[a,b]$ are different multisets. 

\begin{definition}\label{def_2} \rm 
We say that a finite metric space $X$ is {\em determined by the magnitude function} if the multiset of the edge lengths and the combinatorial data are obtained from the magnitude function $M_X(t)$. 
%
\end{definition}

Recall that the cardinality of a finite metric space is obtained from the magnitude function by $\#X=\lim_{t\to\infty}M_X(t)$ (\cite{L13} Proposition 2.2.6, \cite{LW} Theorem 3). 

\begin{definition} \rm 
A finite metric space is, respectively, {\em rationally independent} (ri); {\em $p$-generic} ($\mbox{g}_p$) for $p$ a natural number; or satisfying the {\em strict virtual triangle inequality} (svti) if the following condition is satisfied respectively:
\begin{enumerate}
\item[($\mbox{ri}$)] The edge lengths are linearly independent over $\QQ$. In other words, the sums of edge lengths do not match for different combinations with multiplicity.
\item[($\mbox{g}_p$)] The sums of edge lengths do not match for different combinations of $p$ or fewer edges with multiplicity.
\item[(svti)] $\displaystyle \max_{1\le i,j\le n} d_{ij}<2\min_{1\le k,l\le n, k\ne l} d_{kl}$ \qquad ($n=\#X$).
\end{enumerate}
\end{definition}

Note that the rational independence implies $p$-genericity for any $p$. 
We remark that our conditions are not well suited for graphs. Any graph with graph metric with two edges or more is rationally dependent. Any connected graph except for complete graphs does not satisfy the strict virtual triangle inequality condition. 
A set consisting of $n$-simplex vertices close to a regular $n$-simplex in $\RR^{n+1}$ satisfies the strict virtual triangle inequality condition, whereas a set consisting of the vertices of a needle-shaped tetrahedron does not satisfy this condition.

\begin{theorem}\label{theorem}
A finite metric space $X$ 
is determined by the magnitude function if $X$ satisfies one of the following conditions ($n=\#X$). 

\begin{enumerate}
\item $n=3$. 
\item $X$ is rationally independent. 
\item $X$ is $5$-generic and satisfies the strict virtual triangle inequality condition. 
\item $n=4$ and $X$ satisfies the strict virtual triangle inequality condition. 
\end{enumerate}

\end{theorem}

The proof of the theorem yields 
\begin{proposition}\label{complete_graph}
The complete graph is determined by the magnitude function. 
\end{proposition}
This can be thought of as a discrete version of Proposition 3.3. of \cite{GGL} that a ball is determined by the magnitude function. 
Combining this proposition with the theorem, we can say that in the absence of maximum symmetry, moderate asymmetry is more convenient for identifying spaces. 

\begin{remark}\label{remark}\rm 
Let us introduce Roff-Yoshinaga's realization of the set of isometry classes of unordered $n$-point metric spaces (\cite{RY} Section 2). 
Let $N={n\choose2}$ as before and put 
\[\mathcal{L}_n=\left\{(d_{12},\dots,d_{n-1\,n})\in (\RR_{>0})^N\,\left|\,
\begin{array}{l}
\,d_{ij}+d_{jk}\ge d_{ik} \>\, \forall i,j,k\,(1\le i,j,k\le n), \\[0.5mm]
\mbox{where }\>d_{\la\la}=0,\,  d_{\la\mu}=d_{\mu\la}\>\,\forall\la,\mu
\end{array}
\right.\right\}.\]
The symmetric group $\frak{S}_n$ acts on $\mathcal{L}_n$ by $\sigma\cdot(d_{ij})=(d_{\si(i)\si(j)})$. 
The space of $n$-point metric spaces $\mathcal{M}et_n$ can be identified with $\mathcal{L}_n/\frak{S}_n$. 
We assume that it is equipped with the quotient topology. 

The metric space $X$ is determined by the magnitude function in the sense of Definition \ref{def_2} if and only if $M_X(t)$ determines a point in $\mathcal{M}et_n$. 

Suppose $\mathcal{L}_n$ is equipped with the Lebesgue measure and $\mathcal{M}et_n$ with the image measure. Then the set of rationally dependent $n$-point metric spaces is measure zero in $\mathcal{M}et_n$ since it is a union of countably many codimension one subspaces. 
In this sense, generic finite metric spaces are rationally independent. 
\end{remark}

\section{Proof of the Theorem}

The small scale asymptotics of the magnitude function (Subsection \ref{small-scale_asympt}) is used to prove (1) and (4) of the theorem (i.e. for the cases when n=3 and 4) (Subsections \ref{part1} and \ref{part4}), and the large scale asymptotics (Subsection \ref{large-scale_asympt}) is used to prove (2) of the theorem (i.e. for the general case satisfying rational independence) (Subsection \ref{part2}). 
The proof of Theorem (4) also requires a combinatorial argument. 
Some calculations were checked using Maple.
Note that we have only to prove (2) for $n\ge4$ and (3) for $n\ge5$. 

We first remark that in our cases the similarity matrix $Z_X(t)$ is invertible. 
When $n=3$ the similarity matrix $Z_X(t)$ is invertible for any $t$ $(t>0)$ (\cite{L13} Proposition 2.4.15).
When $n=4$ the similarity matrix $Z_X(t)$ is invertible for any $t$ $(t>0)$ (\cite{M13} Theorem 3.6 (4)). 
For any finite metric space $X$ the similarity matrix $Z_X(t)$ is invertible for all but finitely many $t>0$ (\cite{L13} Proposition 2.2.6 i). 
Therefore 
$z_X(q)$ $(0<q<1)$ is invertible for any sufficiently small $q$.

\bigskip
We introduce sets of subscript tuples that play important roles in combinatorial arguments. 
\begin{definition}\label{def_I} \rm 
Put 
\[
\mathcal{I}_{k}
=\displaystyle \left\{(i_0,i_1,\dots,i_k)\,|\,1\le i_0,\dots,i_k \le n, \, 
i_0\ne i_1\ne \dots \ne i_k \right\}  
\]
for $k\in\bar\NN$, where $\bar\NN=\NN\cup\{0\}$, and 
\begin{equation}\label{eq_def_I}
\begin{array}{rcl}
\mathcal{I}_\otwop&=&\displaystyle \left\{(i,j,k)\,|\,i\ne j\ne k, \,i<k\right\}, \\[1mm]
%

\mathcal{I}_\triangle&=&\{(i,j,k)\,|\,1\le i<j<k\le n\}, \\[2mm]
\mathcal{I}_\otp&=& \displaystyle \left\{(i,j,k,l)\,|\,1\le i,j,k,l \le n, \, 
i\ne j\ne k\ne l, \, i<l\right\}, \hspace{1cm}{}  \\[2mm]
\mathcal{I}_\sotp&=& \displaystyle \left\{(i,j,k,l)\,|\,1\le i,j,k,l \le n, \, 
i,j,k,l \mbox{ are mutually distinct, } \, i<l\right\}, \hspace{1cm}{}  \\[2mm]
\mathcal{I}_{\mbox{\rm \footnotesize disj}}&=&\displaystyle \left\{(i,j,k,l)\,|\,i<j,k<l,\{i,j\}\cap\{k,l\}=\emptyset, \, i<k \right\}, 
%
\end{array}
\notag
\end{equation}
where o.$2$-p., o.$3$-p., s.o.$3$-p. and disj stand for open $2$-path, open $3$-path, simple open $3$-path and disjoint respectively. 
When $n=4$, write $\mathcal{I}_{\mbox{\rm \footnotesize disj}}$ as $\mathcal{I}_{\mbox{\rm \footnotesize opp}}$. 
\end{definition}

\subsection{Asymptotic behavior at small scale}\label{small-scale_asympt} 
Let $\Delta_{X,ij}(t)$ be the $(i,j)$-cofactor of $Z_X(t)$. 
Put 
\[
M\!u_X(t)=\sum_{i,j}\Delta_{X,ij}(t), \quad M\!d_X(t)=\det Z_X(t), 
\] 
then $M_X(t)=M\!u_X(t)/M\!d_X(t)$.

Let $\nu_k=\nu_k(d_{ij})$ and $\de_k=\de_k(d_{ij})$ be the coefficients of series expansion of $M\!u_X(t)$ and $M\!d_X(t)$ respectively: $\displaystyle M\!u_X(t)=\sum_k\nu_k\,t^k, M\!d_X(t)=\sum_l\de_l\,t^l$. 
Roff and Yoshinaga (\cite{RY} Proof of Theorem 2.3) recently proved\footnote{This can also be verified by direct computation when $n=3,4$.} 
\begin{equation}\label{RY_thm23}
\nu_0=\dots =\nu_{n-2}=0,\> \de_0=\dots =\de_{n-2}=0,\> \nu_{n-1}=\de_{n-1}.
\end{equation}
Put $\displaystyle M_\la=\lim_{t\to 0^+}\frac{d^\la}{dt^\la}M_X(t)$ for $\la\ge0$. 
The identities in \eqref{RY_thm23} implies that if $\de_{n-1}\ne0$ then $\displaystyle M_1=\lim_{t\to0^+}M_X'(t)$  is given by 
\begin{equation}\label{M1}
M_1=\frac{\nu_n-\de_n}{\de_{n-1}}.
\end{equation}
%
%

The asymptotic behavior of the magnitude function at small scale will be used when $n=3$ and $4$. 
We assume $n=3$ or $4$ in what follows in this Subsection. 
\begin{proposition}\label{pos_Mpd}
\begin{enumerate}
\item $\de_2$ is positive for any $3$-point space. 
\item $\de_3$ is non-negative for any $4$-point space and positive if strict inequalities hold in all triangle inequalities. 
\item $\de_3$ is positive for any $4$-point metric subspace of the Euclidean space with the standard metric. 
\end{enumerate}
\end{proposition}

\begin{proof} 
(1) Direct calculation shows 
\begin{equation}\label{cd2}
\de_2=\frac12\sum_{\{i,j,k\}=\{1,2,3\}}(d_{jk}+d_{ik}-d_{ij})(d_{ik}+d_{ij}-d_{jk}).
\end{equation}
Among the three terms of the form $d_{jk}+d_{ik}-d_{ij}$ $(i<j)$ at most only one can be $0$. Hence $\de_2>0$. 

\smallskip
(2) Direct calculation shows that $\de_3$ is given by 
\[
-2\sum_{\mathcal{I}_{\mbox{\rm \scriptsize disj}}}\left(d_{ij}^{\,2}d_{kl}+d_{ij}d_{kl}^{\,2}\right)-2\sum_{\mathcal{I}_\triangle}d_{ij}d_{jk}d_{ik}+2\sum_{\mathcal{I}_\scsotp}d_{ij}d_{jk}d_{kl},
\]
and that it is equal to 
\begin{equation}\label{cd3}
\frac16\sum_{\{i,j,k,l\}=\{1,2,3,4\}}(d_{ik}+d_{jk}-d_{ij})(d_{il}+d_{kl}-d_{ik})(d_{ij}+d_{jl}-d_{il}). 
\end{equation}

(3) The above statement implies that $\de_3$ is positive if none of the triangles collapses. 
Suppose that one triangle, say $\triangle P_1P_2P_3$ collapses. 
If $P_4$ is not on the line $L$ through $P_1, P_2$ and $P_3$, then 
\[(d_{24}+d_{12}-d_{14})(d_{34}+d_{23}-d_{24})(d_{14}+d_{13}-d_{34})>0,\]
which implies $\de_3>0$. 
If $P_4$ is on the line $L$, then we may assume without loss of generality that $P_1, P_2, P_3, P_4$ lie on $L$ in this order. Then $\de_3=8d_{12}d_{23}d_{34}>0$. 
\end{proof}

\begin{remark}\rm 
(1) Proposition \ref{pos_Mpd} implies that any $3$-point set and $4$-point set in the Euclidean space is {\sl generic} in the sense of Roff-Yoshinga (Theorem 2.3 of \cite{RY}). 
It gives an alternative direct proof that any $3$-point set and $4$-point set in the Euclidean space has one-point property. 

\smallskip
(2) The strict triangle inequality condition in (2) of Proposition \ref{pos_Mpd} is satisfied if either the strict virtual triangle inequality condition or the rational independence 
condition (ri) is satisfied. 
Therefore, for $4$-point sets, our condition is stronger (i.e. more restrictive) than Roff-Yoshinga's genericity condition. 

\smallskip
(3) There is an example of a $4$-point metric space that makes $\de_3$ zero, for example, a square graph with graph metric. 

\smallskip
(4) A similar equality like \eqref{cd2} or \eqref{cd3} does not hold when $n=5$. The sign of $\de_4$ may change. 
In fact, let $X_{3,2;\ell}$ $(0<\ell\le2)$ be the space obtained by connecting the vertices on the two-point side of the complete bipartite graph $K_{3,2}$ with an edge of length $\ell$ (Figure \ref{K32_l}). 
Then $\de_4=-4\ell(3\ell-4)$. 
Note that when $\de_4$ is negative ($4/3<\de\le2$) the magnitude function behaves like that of $K_{3,2}$ (Figure \ref{mag_K32_l_3over2}). 

Remark that this space cannot be isometrically embedded in Euclidean space. 
\begin{figure}[htbp]
\begin{center}
\begin{minipage}{.35\linewidth}
\begin{center}
\includegraphics[width=0.6\linewidth]{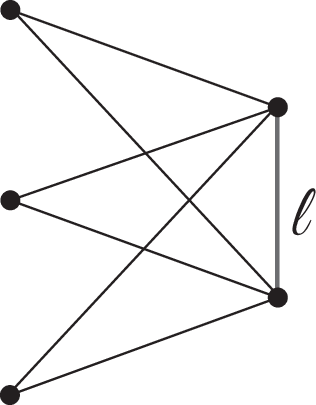}
\caption{$X_{3,2;\ell}$, the complete bipartite graph $K_{3,2}$ attached an edge with length $\ell$}
\label{K32_l}
\end{center}
\end{minipage}
\hskip 0.2cm
\begin{minipage}{.6\linewidth}
\begin{center}
\includegraphics[width=.8\linewidth]{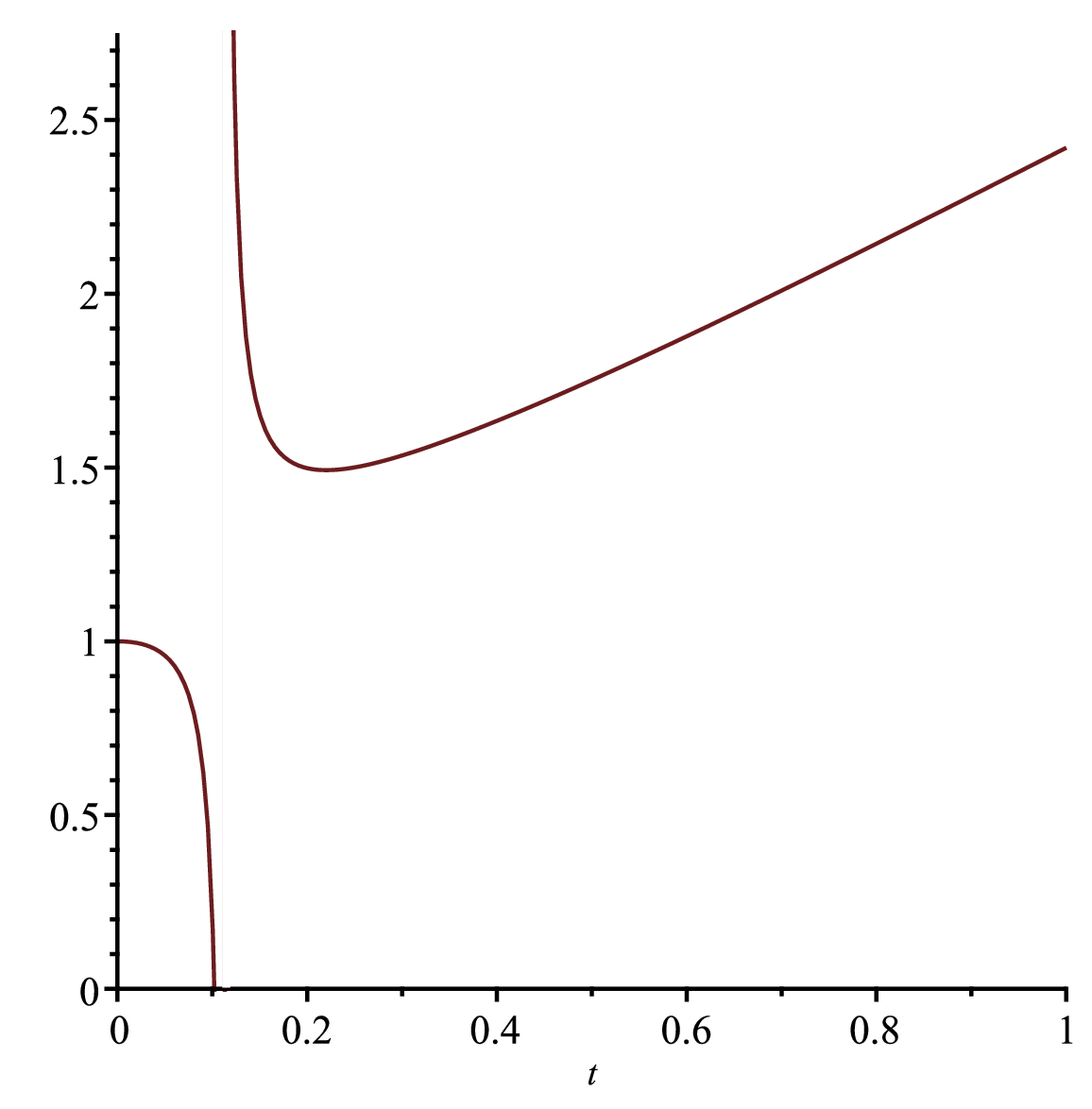}
\caption{The magnitude function of $X_{3,2;\ell}$ when $\ell=3/2$}
\label{mag_K32_l_3over2}
\end{center}
\end{minipage}
\end{center}
\end{figure}
\end{remark}

\subsection{Proof of Part (1) of Theorem \ref{theorem}}\label{part1} 
The proof of this case is carried out by a different way from the other cases. 
We only need the asymptotic behavior of the magnitude function at small scale. 

Assume $n=3$. Put $a=d_{12}, b=d_{13}$ and $c=d_{23}$. 
\begin{lemma}\label{lemma_Mu_Md}
For $\la=1,2$ and $3$, $M_\la=\lim_{t\to 0^+}\frac{d^\la}{dt^\la}M_X(t)$ are given by 
\[
\begin{array}{rcl}
M_1&=&\displaystyle \frac{2abc}{-a^2-b^2-c^2+2ab+2bc+2ca}, \\[6mm]
M_2&=&\displaystyle \frac{2abc(b+c-a)(c+a-b)(a+b-c)}{{(-a^2-b^2-c^2+2ab+2bc+2ca)}^2}, \\[6mm]
M_3&=&\displaystyle \frac{abc\,P(a,b,c)}{{(-a^2-b^2-c^2+2ab+2bc+2ca)}^3},
\end{array}
\]
where 
\[\begin{array}{rcl}
P(a,b,c)&=&\displaystyle a^6+\dots+a^5b+\dots-13a^4b^2-\dots+9a^4bc+\dots+22a^3b^3+\dots \\[2mm]
&&\displaystyle -10a^3b^2c-\dots+30a^2b^2c^2.
\end{array}
\]
\end{lemma}

Since the denominators of $M_\la$ are $\de_{2}^{\la}$ they are positive by Proposition \ref{pos_Mpd}. 

\begin{proof}
If we put\footnote{$\si_\mu$'s make the expressions simpler than the elementary symmetric polynomials of $a,b$ and $c$ do. } $\sigma_\mu=a^\mu+b^\mu+c^\mu$ $(\mu\in\NN)$, we have 
\[\begin{array}{c}
\nu_0=\nu_1=\de_0=\de_1=0,\> \nu_2=\de_2=\si_1^{\,2}-2\si_2,\\[3mm]
\displaystyle \nu_3=-\si_1\si_2+2\si_3, \, \nu_4=\frac{\si_1\si_3}3+\frac{\si_2^{\,2}}4-\frac{7\si_4}6, \, \nu_5=-\frac{\si_1\si_4}{12}-\frac{\si_2\si_3}6+\frac{\si_5}2,\\[3mm]
\displaystyle \nu_3-\de_3=2abc, \> \nu_4-\de_4=-abc\,\si_1, \> \nu_5-\de_5=\frac{abc}4\left(\si_1^{\,2}+\frac{\si_2}3\right).
\end{array}\]
Now $M_1$ is obtained from \eqref{M1} and $M_2$ from 
\[
M_2=2\,\frac{(\nu_4-\de_4)\de_2-(\nu_3-\de_3)\de_3}{\de_{2}^{\,2}}.
\]
$M_3$ is obtained in the same way, although the calculation is more complicated, so we omit the details. 
\end{proof}

\begin{lemma}\label{abc<-M1M2M3}
The edge lengths $a,b$ and $c$ can be obtained from $M_1,M_2$ and $M_3$. 
\end{lemma}

\begin{proof}
Let $s_1, s_2$ and $s_3$ be elementary symmetric polynomials of 
\begin{equation}\label{xyz<->abc}
x=b+c-a, \> y=c+a-b, \> z=a+b-c \,; \notag
\end{equation}
\begin{equation}\label{s123<->xyz}
\left\{\begin{array}{rcl}
s_1&=&x+y+z,\\[1mm]
s_2&=&xy+yz+zx, \\[1mm]
s_3&=&xyz.
\end{array}\right. \notag
\end{equation}
Since $x,y,z\ge0$ and at most one of $x,y$ and $z$ can be $0$, $s_1,s_2>0$ and $s_3\ge0$. 

Since $M_1,M_2$ and $M_3$ are symmetric in $a,b$ and $c$, and the elementary symmetric polynomials of $a, b$ and $c$ can be expressed by $s_1, s_2$ and $s_3$ as 
\begin{equation}\label{abcs1s2s3}
\left\{\begin{array}{rcl}
a+b+c&=&s_1, \\[2mm]
ab+bc+ca&=&\displaystyle \frac{s_1^{\,2}+s_2}4, \\[4mm]
abc&=&\displaystyle \frac{s_1s_2-s_3}8,
\end{array}\right.
\end{equation}
$M_1,M_2$ and $M_3$ can be expressed by $s_1, s_2$ and $s_3$; 
\[
\begin{array}{rcl}
M_1&=&\displaystyle \frac{s_1s_2-s_3}{4s_2},\\[4mm]
M_2&=&\displaystyle \frac{(s_1s_2-s_3)s_3}{4s_2^{\,2}},\\[4mm]
M_3&=&\displaystyle -\frac{(s_1s_2-s_3)\left(s_1^{\,2}s_2^{\,2}+4s_1s_2s_3-3s_2^{\,3}-12s_3^{\,2}\right)}{32s_2^{\,3}}.
\end{array}
\]
Remark that $M_1>0$ since $s_1s_2-s_3=8abc>0$ and $s_2>0$. 
Solving the above equations for $s_1, s_2$ and $s_3$, we obtain 
\begin{equation}\label{s123<->M123}
\begin{array}{rcl}
s_1&=&\displaystyle \frac{4M_1^{\,2}+M_2}{M_1}, \\[6mm]
s_2&=&\displaystyle \frac{16M_1^{\,4}+24M_1^{\,2}M_2+8M_1M_3-7M_2^{\,2}}{3M_1^{\,2}}, \\[6mm]
s_3&=&\displaystyle \frac{M_2\left(16M_1^{\,4}+24M_1^{\,2}M_2+8M_1M_3-7M_2^{\,2}\right)}{3M_1^{\,3}}.
\end{array}
\end{equation}
The edge lengths $a,b$ and $c$ are obtained from $s_1, s_2$ and $s_3$ by \eqref{abcs1s2s3}, and hence from $M_1,M_2$ and $M_3$ by \eqref{s123<->M123}, which completes the proof. 
\end{proof}

\begin{corollary}
A $3$-point metric space $X$ is determined by  $\displaystyle \lim_{t\to 0^+}M^{(\la)}_X(t)$ $(\la=1,2,3)$. 
\end{corollary}

\subsection{Asymptotic behavior at large scale when $n\ge4$}\label{large-scale_asympt} 
%
In this subsection we investigate the asymptotic behavior of the magnitude function $M_X(t)$ at large scale ($t\to+\infty$) as preparation for the proof when $n\ge4$. For this purpose it seems that the use of the formal magnitude $m_X(q)$ $(q=e^{-t})$ would make the description easier to read. 
Assume $n\ge4$ in what follows. 

(i) First we show that a finite rationally independent metric space is determined by the triples of lengths of three consecutive edges that form triangles or open $3$-paths.

Let $X=\{P_1,\dots,P_n\}$. 
Let $\widetilde{\mathcal{S}}_\triangle$ (or $\widetilde{\mathcal{S}}_\otp$) be the multiset of the multisets of lengths of edges forming triangles (or respectively, open $3$-paths): 
\begin{equation}
\begin{array}{rcl}
\widetilde{\mathcal{S}}_\triangle&=&\displaystyle \left[\,[d_{ij},d_{jk},d_{ki}]\,|\,(i,j,k)\in \mathcal{I}_\triangle \right], \notag \\[1mm]
\widetilde{\mathcal{S}}_\otp&=&\displaystyle \left[\,[d_{ij},d_{jk},d_{kl}]\,|\,(i,j,k,l)\in \mathcal{I}_\otp \right].
\notag 
\end{array} 
\end{equation}
Note that $\widetilde{\mathcal{S}}_\triangle$ and $\widetilde{\mathcal{S}}_\otp$ have no information about the subscripts of $d_{\ast\ast}$; in other words, even if we know the triplet of edge lengths, we do not know their vertices.
\begin{lemma}\label{lemma_determine}
A finite $3$-generic metric space is determined by $\widetilde{\mathcal{S}}_\triangle$ and $\widetilde{\mathcal{S}}_\otp$. 
\end{lemma}

\begin{proof}
The $3$-genericity condition ($\mbox{g}_3$) implies that the edge lengths are different from each other. Therefore from $\widetilde{\mathcal{S}}_\triangle$ 
we can obtain the multiset of the edge lengths $\widetilde{\mathcal{S}}_1=[d_{ij}\,|\,1\le i<j\le n]$ (in fact it is a set in this case). 

Let $\widetilde{\mathcal{S}}_\sotp$ (or $\mathcal{S}_\sotp$) be the multiset of the multisets of lengths of edges forming simple open $3$-paths (or respectively, with the information of the length of the middle edge):
\begin{eqnarray}
%
\widetilde{\mathcal{S}}_\sotp&=&\displaystyle \left[\,[d_{ij},d_{jk},d_{kl}]\,|\,(i,j,k,l)\in \mathcal{I}_\sotp \right],
\notag \\[1mm]
\mathcal{S}_\sotp&=&\displaystyle \left[(d_{jk}, [d_{ij},d_{kl}])\,|\,(i,j,k,l)\in \mathcal{I}_\sotp \right]. \notag
\end{eqnarray}

First remark that $\widetilde{\mathcal{S}}_\sotp$ can be obtained from $\widetilde{\mathcal{S}}_\otp$ by removing multisets with duplications. 

Next remark that the data of $\widetilde{\mathcal{S}}_\triangle$ and $\widetilde{\mathcal{S}}_\sotp$ produce $\mathcal{S}_\sotp$, namely, the information of the middle edges of simple open $3$-paths can be obtained from $\widetilde{\mathcal{S}}_\triangle$ and $\widetilde{\mathcal{S}}_\sotp$. 
This is because the two edges at the ends of a simple open $3$-path cannot form a triangle with another edge. 

Finally we show that a finite $3$-generic metric space is determined by $\widetilde{\mathcal{S}}_1, \widetilde{\mathcal{S}}_\triangle$ and $\mathcal{S}_\sotp$. Suppose we have the data of $\widetilde{\mathcal{S}}_1, \widetilde{\mathcal{S}}_\triangle$ and $\mathcal{S}_\sotp$. Choose a triangle and an edge of it. 
We may label the three vertices $P_1, P_2$ and $P_3$ so that the edge we selected is $P_1P_2$. Let $\ell_\a=d(P_2,P_3), \ell_\beta=d(P_3,P_1)$ and $\ell_\ga=d(P_1,P_2)$. 
There are still $n-3$ triangles containing the edge $P_1P_2$. 
There are two ways to attach each triangle to $P_1P_2$, but one is determined from the information of open 3-paths and middle edges as follows. 
Suppose $[\ell_\la,\ell_\mu,\ell_\ga]\in\widetilde{\mathcal{S}}_\triangle$. Let the remaining vertex be $P_4$, say. 
Note that both $\ell_\la, \ell_\mu, \ell_\a$ and $\ell_\la, \ell_\mu, \ell_\beta$ form open $3$-paths. 
If the middle edge of $\ell_\la, \ell_\mu, \ell_\a$ is $\ell_\la$ then the triangle $\triangle P_4P_1P_2$ is attached to the edge $P_1P_2$ in a way that $d(P_4,P_1)=\ell_\mu$ and $d(P_4,P_2)=\ell_\la$, and if not the other way. 

After attaching the remaining $n-4$ triangles to $P_1P_2$, label the remaining vertices $P_5,\dots,P_n$. The lengths $d(P_i,P_1), d(P_i,P_2)$ $(4\le i \le n)$ are determined by the procedure described above. 
The length $d(P_i,P_j)$ $(3\le i<j\le n)$ is determined as the unique element $\ell$ in $\widetilde{\mathcal{S}}_1$ such that $[\ell, d(P_i,P_1), d(P_j,P_1)]$ is an element of $\widetilde{\mathcal{S}}_\triangle$. 
\end{proof}

\medskip
(ii) Next we prepare a proposition which we will use to get a multiset consisting of the sums of lengths of edges forming triangles and a multiset consisting of the sums of lengths of edges forming open $3$-paths from the formal magnitude $m_X(q)$.

\begin{proposition}\label{Leinster}{\rm (Leinster \cite{L19})} 
The formal magnitude of a finite metric space $X$ is given by 
\begin{equation}\label{eqL}
m_X(q)=\sum_{k=0}^\infty(-1)^k\sum_{(i_0,\dots,i_k)\in\mathcal{I}_k}q^{d_{i_0i_1}+\dots+d_{i_{k-1}i_k}}.
\end{equation}
\end{proposition}
In fact it was proved in Proposition 3.9 of \cite{L19} for graphs, and the proof given there works for any finite metric space as well (cf. formula (1) of \cite{HW} and Corollary 7.15 of \cite{LS}). 

\begin{definition}\label{def_d-index}  \rm 
\begin{enumerate}
\item 
Let $\mathcal{P}_{all}$ be the commutative monoid generated by $\widetilde{\mathcal{S}}_1=[d_{ij}]$;
\begin{equation}\label{P_all}
\mathcal{P}_{all}=\Big\{\sum_{i<j}a_{ij}d_{ij}\,|\,a_{ij}\in\bar\NN
\Big\}.
\end{equation}
\item Let $\mathcal{P}$ $(\mathcal{P}\subset\mathcal{P}_{all})$ be the set of positive exponents that appear in $m_X(q)$. 
%
\item 
We define the {\em $d$-index} of a term $q^{\,\sum_{i<j}a_{ij}d_{ij}}$ to be $\sum_{i<j}a_{ij}$. 
\end{enumerate}
\end{definition}
%

Let $m_{X,3}(q)$ be the sum of all the terms in $m_X(q)$ with the $d$-index less than or equal to $3$; 
\[
m_{X,3}(q)=\sum_{k=0}^3(-1)^k\sum_{(i_0,\dots,i_k)\in\mathcal{I}_k}q^{d_{i_0i_1}+\dots+d_{i_{k-1}i_k}}.
\]
By dividing $2$- and $3$-paths into closed paths and open paths we obtain 
\begin{eqnarray}
m_{X,3}(q)&=&\displaystyle n-2\sum_{i<j}q^{d_{ij}}+2\sum_{i<j}q^{2d_{ij}}
+2\sum_{\mathcal{I}_\scotwop} q^{d_{ij}+d_{jk}}
%
-6\sum_{\mathcal{I}_\triangle} q^{d_{ij}+d_{jk}+d_{ki}} 
-2\sum_{\mathcal{I}_\scotp} q^{d_{ij}+d_{jk}+d_{kl}}.
%
\label{m_3}
\end{eqnarray}
%

\medskip
(iii) Finally we give a lemma to obtain the exponents and coefficients 
of $m_X(q)$ 
from the magnitude function $M_X(t)$. 

\begin{lemma}\label{lemma_P} 
Suppose $m_X(q)$ is expressed as 
\[
m_X(q)=\sum_{m=0}^\infty a_mq^{\a_m} \quad (a_m\in\RR,\,\{\a_0,\a_1.\dots\}=\mathcal{P},\,\a_0<\a_1<\dots ).
\]
Then $\a_0=0$ and $\displaystyle a_0=\lim_{t\to+\infty}M_X(t)=\#X$, and $\a_m$ and $a_m$ are given inductively by 
\[
\begin{array}{rcl}
\a_m&=&\displaystyle \lim_{t\to+\infty}\frac{\displaystyle \left|\log \left(M_X(t)-\sum_{i=0}^{m-1}a_i\, e^{-t\a_i}\right)\right|}{t}, \\[6mm]
a_m&=&\displaystyle \lim_{t\to+\infty} e^{t\a_m}\left(M_X(t)-\sum_{i=0}^{m-1}a_i\, e^{-t\a_i}\right).
\end{array}
\]

\end{lemma}

\subsection{Proof of Part (2) of Theorem \ref{theorem}}\label{part2} 
%
By Lemma \ref{lemma_P} we obtain the multiset $\mathcal{P}$ from $M_X(t)$. 
From $\mathcal{P}$ we can obtain the multiset of the edge lengths $\widetilde{\mathcal{P}}_1=[d_{12},\dots,d_{n-1n}]$ as follows. 
Remark that the rational independence condition (ri) implies that $d_{ij}$'s are all different from each other, hence $\widetilde{\mathcal{P}}_1$ is an ordinary set in fact. 
Define $\ell_1,\dots,\ell_N$ inductively, where $N={n\choose2}$, by 
\begin{equation}\label{how_to_determine_ell}
\begin{array}{l}
\displaystyle \ell_1=\min\mathcal{P}, \\[2mm]
\displaystyle \ell_2=\min\left(\mathcal{P}\setminus \{\tau\,\ell_1\,|\,\tau\in\NN\}\right), \\[2mm] 
\displaystyle \ell_3=\min\left(\mathcal{P}\setminus \{\tau_1\,\ell_1+\tau_2\,\ell_2\,|\,\tau_1,\tau_2\in\NN\cup\{0\}\}\right), \\[2mm]
\dots
\end{array}
\notag
\end{equation}
Then the (multi)set of edge lengths $\widetilde{\mathcal{P}}_1$ is given by the (multi)set $[\ell_1,\dots,\ell_N]$. 

Put 
\[\begin{array}{rcl}
\widetilde{\mathcal{S}}_p&=&\displaystyle [\,[\ell_{\a_{i_1}},\dots,\ell_{\a_{i_p}}]\,|\,1\le \a_{i_1}\le\dots\le\a_{i_p}\le N] \qquad (p\in\NN), \\[2mm]
\widetilde{\mathcal{S}}&=&\displaystyle \bigcup_{p\in\NN}\widetilde{\mathcal{S}}_p.
\end{array}
\]
Note that $\widetilde{\mathcal{S}}_1=\widetilde{\mathcal{P}}_1$ and that $\widetilde{\mathcal{S}}_3\supset \widetilde{\mathcal{S}}_\triangle, \widetilde{\mathcal{S}}_\otp$. 
The rational independence condition (ri) implies that a map $\Sigma\colon\widetilde{\mathcal{S}}\to\RR$ given by 
\[\Sigma\left(\,[\ell_{\a_{i_1}},\dots,\ell_{\a_{i_p}}]\,\right)=\ell_{\a_{i_1}}+\dots+\ell_{\a_{i_p}}\]
is injective. 
Since the ``triangle'' terms and the ``{open $3$-path}'' terms in \eqref{m_3} have different coefficients, by comparing $\mathcal{P}$ and $\Sigma(\widetilde{\mathcal{S}}_3)$ 
we obtain $\Sigma(\widetilde{\mathcal{S}}_\triangle)$ and $\Sigma(\widetilde{\mathcal{S}}_\otp)$.
Since $\Sigma$ is injective, the conclusion follows from Lemma \ref{lemma_determine}. 

\subsection{Proof of Part (3) of Theorem \ref{theorem}}\label{part3} 
The proof is almost the same as the previous case.

The strict virtual triangle inequality condition implies 
\[\max_{1\le\a\le N} \ell_\a<\min \bigcup_{p\ge 2}\,\Sigma(\widetilde{\mathcal{S}}_p),\]
hence $\ell_\a$ $(1\le\a\le N)$ can be determined as the first $N$ smallest numbers of $\mathcal{P}$. 
Since the strict virtual triangle inequality condition implies $\Sigma(\widetilde{\mathcal{S}}_3)\,\cap\,\cup_{p\ge 6}\,\Sigma(\widetilde{\mathcal{S}}_p)=\emptyset$ and the $5$-genericity condition ($\mbox{g}_5$) implies $\Sigma\colon \cup_{p\le 5}\Sigma(\widetilde{\mathcal{S}}_p)\to\RR$ is injective, we obtain $\Sigma(\widetilde{\mathcal{S}}_\triangle)$ and $\Sigma(\widetilde{\mathcal{S}}_\otp)$. 
The rest of the proof is same as in the previous case.

\subsection{Proof of Part (4) of Theorem \ref{theorem}}\label{part4} 

When the number of points is four, each edge has exactly one disjoint edge, which we call the {\em opposite} edge, and accordingly, $\mathcal{I}_{\mbox{\footnotesize disj}}$ will be denoted by $\mathcal{I}_{\mbox{\footnotesize opp}}$ hereafter. 
Note that a four point set has three pairs of opposite edges. 

Without the rational independence condition, it can happen that the combination of edges cannot be determined from $m_X(q)$, as was the case in the example of graphs in the Introduction (Figure \ref{two_graphs}). 
This complicates the proof. 

Suppose we know the generalized formal power series expression of $m_X(q)$. The proof consists of the following four steps. 
\begin{enumerate}
\item[1.] The multiset of the edge lengths is determined. 
\item[2.] The multiset of the sums of lengths of pairs of opposite edges is determined. 
\item[3.] The combination of opposite edges that give the sums mentioned above is determined. 
\item[4.] One of the two possible ``tetrahedra'' is determined. 
\end{enumerate}

Step 1. The strict virtual triangle inequality condition implies that the multiset of the edge lengths $\widetilde{\mathcal{P}}_1=[\ell_1,\dots,\ell_6]$ $(\ell_1\le \dots \le \ell_6)$ is obtained by taking $N$ numbers with multiplicity from $\mathcal{P}$, increasing from the smallest. 
The multiplicity can be determined by the coefficient of $q^{d_{ij}}$ in $m_X(q)$ divided by $-2$.

\medskip
Step 2. Put for $p=1,2$ and $3$ 
\[
\si_p(q)=\sum_{i<j}q^{p\,d_{ij}}=q^{p\,\ell_1}+\dots+q^{p\,\ell_6}.
\]
Define $f(q)$ by modifying $m_X(q)$ as 
\begin{equation}\label{def_f}
f=m_X-4+2\si_1-\si_1^2-\si_2+\si_1\si_2+\frac13\si_1^3+\frac23\si_3.
\end{equation}
Assume that the terms of $f(q)$ are in order of increasing power.
Let $f_3(q)$ be the sum of all the terms appearing in $f(q)$ with $d$-index less than or equal to $3$. 
Since $-2\sum_{\mathcal{I}_\scotp} q^{d_{ij}+d_{jk}+d_{kl}}$ in \eqref{m_3} is equal to 
\[
\displaystyle 
-2\Big(\sum_{i<j}q^{2d_{ij}}\Big)\Big(\sum_{k<l}q^{d_{kl}}\Big) 
%
%
-2\sum_{\mathcal{I}_\scsotp} q^{d_{ij}+d_{jk}+d_{kl}}
+2\sum_{\mathcal{I}_{\mbox{\rm \scriptsize opp}}} \left(q^{2d_{ij}+d_{kl}}+q^{d_{ij}+2d_{kl}}\right),
\]
$f_3(q)$ is given by 
\begin{equation}\label{f_3}
\begin{array}{rcl}
f_3(q)&=&\displaystyle -2\sum_{\mathcal{I}_{\mbox{\scriptsize  opp}}} q^{d_{ij}+d_{kl}} \\[6mm]
&&\displaystyle 
-4\sum_{\mathcal{I}_\triangle} q^{d_{ij}+d_{jk}+d_{ki}}
+2\sum_{\mathcal{I}_{\mbox{\scriptsize  vtx}}} q^{d_{ij}+d_{ik}+d_{il}}
+2\sum_{\mathcal{I}_{\mbox{\scriptsize  opp}}} \left(q^{2d_{ij}+d_{kl}}+q^{d_{ij}+2d_{kl}}\right),
\end{array}
\end{equation}
where
\[
\mathcal{I}_{\mbox{\footnotesize  vtx}}=\{(i,j,k,l)\,|\,1\le i\le4,\,j<k<l,\,\{j,k,l\}=\{1,2,3,4\}\setminus\{i\}\}.
\]

\begin{lemma}\label{claim} 
Assume $(i,j,k,l)\in \mathcal{I}_{\mbox{\rm \footnotesize opp}}$. Then the following holds. 
\begin{enumerate}
\item $d_{ij}+d_{kl}<d_{i'j'}+d_{i'k'}+d_{i'l'}$ for any $(i',j',k',l')\in \mathcal{I}_{\triangle}$. 
\item $d_{ij}+d_{kl}<d_{i'j'}+d_{i'k'}+d_{i'l'}$ for any $(i',j',k',l')\in \mathcal{I}_{\mbox{\rm \footnotesize vtx}}$. 
\item $d_{ij}+d_{kl}$ is smaller than the exponent of any term of $f(q)$ with $d$-index $4$ or more. 
\end{enumerate}
\end{lemma} 

\begin{proof}
(1) Any pair of opposite edges has exactly one edge in common with three edges of any triangle. The remaining inequality is a consequence of the strict virtual triangle inequality. 

(2) Any pair of opposite edges has exactly one edge in common with three edges having one common vertex.  

(3) Consequence of the strict virtual triangle inequality. 
\end{proof}

Put 
\[
\begin{array}{rcl}
\widetilde{\mathcal{P}}_{\rm opp,2}&=&\displaystyle [ d_{ij}+d_{kl} \,|\, (i,j,k,l)\in \mathcal{I}_{\rm opp} ], \\[2mm]
\widetilde{\mathcal{P}}_{\rm opp,3}&=&\displaystyle [ 2d_{ij}+d_{kl}, d_{ij}+2d_{kl} \,|\, (i,j,k,l)\in \mathcal{I}_{\rm opp} ]. 
\end{array}
\]

\begin{lemma}\label{lemma_2-ii}
Let $[\ell_1,\dots,\ell_6]$ be the multiset of edge lengths. 
If two elements of $\widetilde{\mathcal{P}}_{\rm opp,2}$ appear in $\widetilde{\mathcal{P}}_{\rm opp,3}$, 
we can find $\a,\beta,\ga,\de,\la,\mu$ 
that satisfy the following; 
\begin{enumerate}
\item $\{\a,\beta,\ga,\de,\la,\mu\}=\{1,2,3,4,5,6\}$, 

$\ell_\a\le\ell_\beta, \ell_\ga\le\ell_\de, \ell_\la\le\ell_\mu$, $\ell_\a+\ell_\beta \le \ell_\ga+\ell_\de \le \ell_\la+\ell_\mu$,
\item $\ell_\a+\ell_\beta=(\ell_1+\dots+\ell_6)/4$,
\item $\ell_\ga+\ell_\de=2\ell_\a+\ell_\beta$ and $\ell_\la+\ell_\mu=\ell_\a+2\ell_\beta$.
\end{enumerate}
\end{lemma}

\begin{proof} 
First note that (2) is a consequence of (3). 

We may assume without loss of generality that $\widetilde{\mathcal{P}}_{\rm opp,2}$ is given by $[\ell_\a+\ell_\beta, \ell_\ga+\ell_\de, \ell_\la+\ell_\mu]$ with condition (1) above. 
Since the smallest element $\ell_\a+\ell_\beta$ cannot appear in $\widetilde{\mathcal{P}}_{\rm opp,3}$, it is enough to show that the case 
\begin{equation}\label{cannot_happen}
\begin{array}{rcl}
\ell_\ga+\ell_\de&=&2\ell_\a+\ell_\beta \quad \mbox{ or }\quad \ell_\a+2\ell_\beta, \\[2mm]
\ell_\la+\ell_\mu&=&2\ell_\ga+\ell_\de \quad \mbox{ or }\quad \ell_\ga+2\ell_\de
\end{array}
\end{equation}
cannot happen. 

Assume \eqref{cannot_happen}. 
By the strict virtual triangle inequality condition we have $\ell_\ga>(\ell_\ga+\ell_\de)/3$, which implies $\ell_\la+\ell_\mu>\frac43(\ell_\ga+\ell_\de)$. 
Put $t=\ell_\a/(\ell_\a+\ell_\beta)$. The assumption $\ell_\a\le\ell_\beta$ and the strict virtual triangle inequality condition imply $1/3<t\le 1/2$. Therefore we have
\[
\ell_\la+\ell_\mu>\frac43(\ell_\ga+\ell_\de)\ge\frac43(t+1)(\ell_\a+\ell_\beta)
=\frac{4(t+1)}{3t}\ell_\a, 
\]
which implies 
\[
\ell_\mu\ge\frac12(\ell_\la+\ell_\mu)>\frac23\left(1+\frac1t\right)\ell_\a.
\]
On the other hand, since $t\le1/2$, it means $\ell_\mu>2\ell_\a$, which contradicts the strict virtual triangle inequality condition. 
\end{proof}

We remark that if (1), (2) and (3) of Lemma \ref{lemma_2-ii} are satisfied then $\ell_\la+\ell_\mu$ cannot be equal to either $2\ell_\ga+\ell_\de$ or $\ell_\ga+2\ell_\de$ by the strict virtual triangle inequality condition.

\begin{proposition}
The multiset of the sums of edge lengths of pairs of opposite edges, 
$\widetilde{\mathcal{P}}_{\rm opp,2}=\displaystyle [ d_{ij}+d_{kl} \,|\, (i,j,k,l)\in \mathcal{I}_{\rm opp} ] $ can be obtained from the generalized formal power series expression of $m_X(q)$. 
\end{proposition}

\begin{proof}
Since it is impossible that all the three elements of $\widetilde{\mathcal{P}}_{\rm opp,2}$ appear in $\widetilde{\mathcal{P}}_{\rm opp,3}$, there are only two possibilities:

\medskip\noindent
Case 1. At most one element of $\widetilde{\mathcal{P}}_{\rm opp,2}$ appears in $\widetilde{\mathcal{P}}_{\rm opp,3}$. 

\medskip\noindent
Case 2. Two elements of $\widetilde{\mathcal{P}}_{\rm opp,2}$ appear in $\widetilde{\mathcal{P}}_{\rm opp,3}$. 

\medskip
First we show that if either Case 1 or Case 2 is known in advance, in each case $\widetilde{\mathcal{P}}_{\rm opp,2}$ can be obtained from $f(q)$. 
Assume $f(q)$ is arranged in increasing powers of $q$. 

Case 1. Lemma \ref{claim} implies that at least two of the terms $q^{d_{ij}+d_{kl}}$, where $(i,j,k,l)\in \mathcal{I}_{\mbox{\footnotesize opp}}$, in \eqref{f_3} survive, i.e. the coefficients do not cancel out. 
Take the first term in $f(q)$ with coefficient $-4$ if exists or if not the first two terms with coefficient $-2$. 
Then the exponent(s) give(s) two elements of $\widetilde{\mathcal{P}}_{\rm opp,2}$. The remaining one element can be obtained by subtracting the sum of the two from $\ell_1+\dots+\ell_6$.  

Case 2. The first term of $f(q)$ has coefficient $-2$ and exponent $\ell_a+\ell_\beta$ in Lemma \ref{lemma_2-ii}. 
The strict virtual triangle inequality condition implies that if there are two pairs $\{\a,\beta\}$ $(\a\ne\beta)$ and $\{\a',\beta'\}$ $(\a'\ne\beta')$ with 
\[\ell_\a+\ell_\beta=\ell_{\a'}+\ell_{\beta'}=\frac{\ell_1+\dots+\ell_6}4\]
then $\#\{\a,\beta,\a',\beta'\}\le3$, namely, the values of $\ell_\a$ and $\ell_\beta$ in Lemma \ref{lemma_2-ii} are fixed. 
Then the remaining two elements of $\widetilde{\mathcal{P}}_{\rm opp,2}$ can be obtained by $2\ell_\a+\ell_\beta$ and $\ell_\a+2\ell_\beta$.

\medskip
Next we show that it can be determined whether Case 1 or Case 2 is occurring from the information of $f(q)$. In fact, Case 2 can occur if and only if the following conditions are all satisfied. 

\medskip\noindent
The coefficient of the first term of $f(q)$ is equal to $-2$. 
The exponent of this term is given by $\ell_\a+\ell_\beta$ for some $\a$ and $\beta$ $(\a\ne\beta)$. 
We can choose $\ga,\de,\la,\mu$ such that $\{\ga,\de,\la,\mu\}=\{1,\dots,6\}\setminus\{\a,\beta\}$ and  $\ell_\ga+\ell_\de=2\ell_\a+\ell_\beta$ and $\ell_\la+\ell_\mu=\ell_\a+2\ell_\beta$ hold. 
There is no term with exponent $\ell_\ga+\ell_\la, \ell_\ga+\ell_\mu, \ell_\de+\ell_\la$ or $\ell_\de+\ell_\mu$ with coefficient $-2$ or $-4$ in $f(q)$. 
\end{proof}

\medskip
Step 3. 
Given a multiset of edge lengths $\widetilde{\mathcal{P}}_1=[\ell_1, \dots, \ell_6]$  $(\ell_1\le \dots \le \ell_6)$ and a multiset of the sums of opposite edges $\widetilde{\mathcal{P}}_{\rm opp,2}$, if the combination of opposite edges that realizes the sums is not unique, there are only the following two cases.  

\begin{enumerate}
\item[1.] $\ell_2-\ell_1=\ell_5-\ell_4$ and $\ell_3-\ell_2=\ell_6-\ell_5$. Then 
\[
(\ell_1+\ell_5, \ell_3+\ell_4, \ell_2+\ell_6)=(\ell_2+\ell_4, \ell_1+\ell_6, \ell_3+\ell_5).
\] 
\item[2.] $\ell_2-\ell_1=\ell_4-\ell_3=\ell_6-\ell_5$. Then 
\[
(\ell_1+\ell_4, \ell_2+\ell_5,\ell_3+\ell_6)=(\ell_2+\ell_3,\ell_1+\ell_6,\ell_4+\ell_5).
\]
\end{enumerate}

Case 1. Let ${\rm COMB1}=[\ell_1+\ell_5, \ell_3+\ell_4, \ell_2+\ell_6]$ and ${\rm COMB2}=[\ell_2+\ell_4, \ell_1+\ell_6, \ell_3+\ell_5]$. 
Each combination has two possible configurations as illustrated in Figure \ref{four_config}. 
\begin{figure}[htbp]
\begin{center}
\includegraphics[width=.8\linewidth]{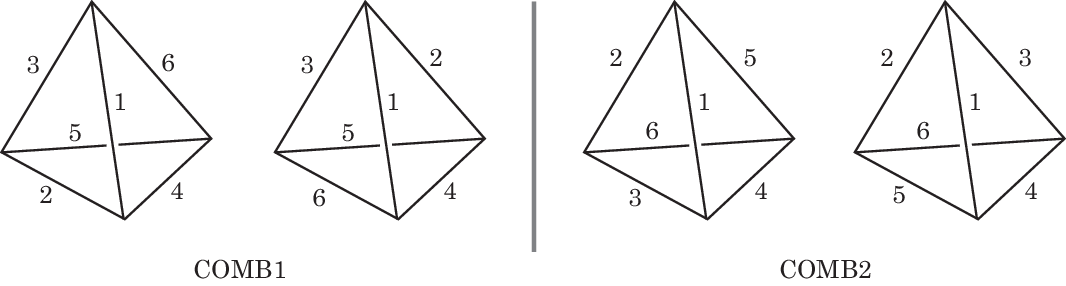}
\caption{Four possible configurations. $i$ stands for $\ell_i$.}
\label{four_config}
\end{center}
\end{figure}

Define $g(q)$ by $g(q)=f(q)+2\sum_{\mathcal{I}_{\mbox{\scriptsize opp}}} q^{d_{ij}+d_{kl}}$, where $f$ is given by \eqref{def_f}, and let $g_3(q)$ be the sum of all the terms in $g(q)$ with $d$-index less than or equal to $3$. Then \eqref{f_3} implies 
\[
g_3(q)=
-4\sum_{\mathcal{I}_\triangle} q^{d_{ij}+d_{jk}+d_{ki}}
2\sum_{\mathcal{I}_{\mbox{\scriptsize vtx}}} q^{d_{ij}+d_{ik}+d_{il}}
+2\sum_{\mathcal{I}_{\mbox{\scriptsize opp}}} \left(q^{2d_{ij}+d_{kl}}+q^{d_{ij}+2d_{kl}}\right).
\]
In each configuration in Figure \ref{four_config}, three edges labelled $1,2$ and $3$ form either a triangle or a ``Y shape'', hence $\ell_1+\ell_2+\ell_3$ always appears as the exponent of either a ``triangle'' or a ``vertex'' term in $g_3(q)$ as long as the coefficient does not cancel out with the coefficients of other terms.

\begin{lemma}
$\ell_1+\ell_2+\ell_3$ is the smallest exponent that appears in $g(q)$. 
\end{lemma}

Put $u=\ell_4-\ell_3, v=\ell_2-\ell_1=\ell_5-\ell_4, w=\ell_3-\ell_2=\ell_6-\ell_5$ 
$(u,v,w\ge0)$. 

\begin{proof}
First note that $\ell_1+\ell_2+\ell_3$ is the minimum of $\ell_i+\ell_j+\ell_k$ with $i\ne j\ne k \ne i$, hence it gives the minimum of 
\[\{d_{ij}+d_{ik}+d_{il}\,|\,(i,j,k,l)\in \mathcal{I}_{\mbox{\footnotesize vtx}}\}\cup\{d_{ij}+d_{jk}+d_{ki}\,|\,(i,j,k)\in\mathcal{I}_\triangle\}.\]
Next since 
\[
\begin{array}{l}
COMB1\colon \,
\left\{
\begin{array}{rcl}
2\ell_1+\ell_5&=&\ell_1+\ell_2+\ell_3+u, \\
2\ell_2+\ell_6&=&\ell_1+\ell_2+\ell_3+u+2v+w, \\
2\ell_3+\ell_4&=&\ell_1+\ell_2+\ell_3+u+v+2w,
\end{array}
\right. \\[8mm]
COMB2\colon \,
\left\{
\begin{array}{rcl}
2\ell_2+\ell_4&=&\ell_1+\ell_2+\ell_3+u+v, \\
2\ell_1+\ell_6&=&\ell_1+\ell_2+\ell_3+u+w, \\
2\ell_3+\ell_5&=&\ell_1+\ell_2+\ell_3+u+2v+2w,
\end{array}
\right.
\end{array}
\]
we have 
\[
\ell_1+\ell_2+\ell_3 \le \min\{2d_{ij}+d_{kl}, d_{ij}+2d_{kl}\,|\,(i,j,k,l)\in \mathcal{I}_{\mbox{\footnotesize opp}} \}.
\]
Finally, since $\ell_1+\ell_2+\ell_3\le 2\ell_1+\ell_5$ and the strict virtual triangle inequality implies 
\[
2\ell_1+\ell_5 < 4\ell_1=\min\{\mbox{exponents that appear in $g(q)$ with $d$-index $\ge4$}\},
\]
the conclusion follows. 
\end{proof}

\begin{corollary}
Assume $v=\ell_2-\ell_1=\ell_5-\ell_4>0$ and $w=\ell_3-\ell_2=\ell_6-\ell_5>0$. 
The coefficient of $q^{\ell_1+\ell_2+\ell_3}$ in $g_3(q)$ cancels out if and only if COMBI1 occurs and $u=\ell_4-\ell_3=0$. 
\end{corollary}

In this case $\ell_1+\ell_2+\ell_3=\ell_1+\ell_2+\ell_4=2\ell_1+\ell_5$, one of $q^{\ell_1+\ell_2+\ell_3}$ and $q^{\ell_1+\ell_2+\ell_4}$ has coefficient $-4$ and the other $2$, and $q^{2\ell_1+\ell_5}$ has coefficient $2$ in $g_3(q)$.

\begin{proposition}
In Case 1 the combination of opposite edges can be determined from the information of exponents of $q$ in $g(q)$. 
\end{proposition}

\begin{proof}
Suppose $w=0$. Then $\ell_2=\ell_3$ and $\ell_5=\ell_6$, which implies COMB1$=$COMB2. Similarly, $v=0$ also implies COMB1$=$COMB2. 
Therefore, we have only to consider the case when $v$ and $w$ are both positive. 

When $u$ is positive, COMB1 occurs if the next smallest exponent in $g(q)$ is $2\ell_1+\ell_5=\ell_1+\ell_2+\ell_4$, COMB2 otherwise (i.e. if the next smallest exponent in $g(q)$ is either $2\ell_2+\ell_4=\ell_1+\ell_2+\ell_5$ or $2\ell_1+\ell_6=\ell_1+\ell_3+\ell_4$). 
Remark that the coefficients of $q^{2\ell_1+\ell_5}$ etc. do not vanish. 

When $u=0$, COMB1 if the coefficient of $q^{\ell_1+\ell_2+\ell_3}$ is $0$, COMB2 if it is $-4$ or $2$. 
\end{proof}

The proof for Case 2 can be carried out similarly. 

\medskip
Step 4. Assume that the combination of opposite edges is known, namely we know a multiset of pairs of opposite edge lengths 
\[
[\,[\ell_\a, \ell_\beta], [\ell_\ga, \ell_\de], [\ell_\la, \ell_\mu]\,|\,\mbox{condition (1) in Lemma \ref{lemma_2-ii}}\,]
\]
At this point, there are at most two possibilities for an isometric class of four points, by swapping one of the pairs of opposite edges. 
By permutation of indices, we may assume without loss of generality that $d_{14}$ and $d_{23}$ are the lengths of the edges to be swapped. 

Recall $Z_X(t)=\left(\exp(-td_{ij})\right)_{i,j}$ and the magnitude function $M_X(t)$ is the sum of all the entries of $Z_X(t)^{-1}$. 
The strict virtual triangle inequality condition implies that the strict triangle inequality condition in Proposition \ref{pos_Mpd} (2) is satisfied, and hence $\de_3$ is positive. 
Direct computation of \eqref{M1} shows that $M_1=\lim_{t\to0^+}M_X'(t)$ is given by\footnote{We remark that the denominator $\de_3$ was already given in the proof of Proposition \ref{pos_Mpd} (2). } 
\begin{equation}\label{M'(t)at0}
\frac{-d_{12}^{\,2}d_{34}^{\,2}-d_{13}^{\,2}d_{24}^{\,2}-d_{14}^{\,2}d_{23}^{\,2}+2d_{13}d_{14}d_{23}d_{24}+2d_{12}d_{14}d_{23}d_{34}+2d_{12}d_{13}d_{24}d_{34}}{-2\sum_{\mathcal{I}_{\mbox{\rm \scriptsize opp}}}\left(d_{ij}^{\,2}d_{kl}+d_{ij}d_{kl}^{\,2}\right)-2\sum_{\mathcal{I}_\triangle}d_{ij}d_{jk}d_{ik}+2\sum_{\mathcal{I}_\scsotp}d_{ij}d_{jk}d_{kl}} .
\end{equation}

\begin{lemma}
{\rm (1)} Under strict virtual triangle inequality condition, the numerator is positive. 

\medskip\noindent
{\rm (2)} The numerator is symmetric in $d_{14}$ and $d_{23}$. 

\medskip\noindent
{\rm (3)} The difference in the denominator caused by exchanging $d_{14}$ and $d_{23}$ is given by 
\begin{equation}\label{dd}
\begin{array}{l}
\displaystyle \de_3(d_{12},d_{13},d_{14},d_{23},d_{24},d_{34})-\de_3(d_{12},d_{13},d_{23},d_{14},d_{24},d_{34}) 
%
\displaystyle =
2(d_{12}-d_{34})(d_{13}-d_{24})(d_{14}-d_{23}).
\end{array}
\end{equation}
\end{lemma}

\begin{proof}
(1) We may assume without loss of generality that $d_{12}d_{34}\le d_{13}d_{24} \le d_{14}d_{23}$. Putting $s=d_{13}d_{24}/d_{12}d_{34}, t=d_{14}d_{23}/d_{12}d_{34}$, the numerator can be expressed as $\left(4s-((t-s)-1)^2\right)d_{12}^{\,2}d_{34}^{\,2}$, which is positive since the strict virtual triangle inequality condition implies $1\le s\le t<4$. 

\medskip
(2) Obvious. 

\medskip
(3) By direct computation. Note that the difference comes from the $\mathcal{I}_\triangle$-terms since the terms coming from $\mathcal{I}_\sotp$ cancel each other. 
\end{proof}

Note that if $\eqref{dd}$ vanishes, there is only one possible configuration up to isometry. Therefore we have 

\begin{corollary} 
Under the strict virtual triangle inequality condition, if the three pairs of opposite edges are known, then the four-point set is determined by the magnitude function.
\end{corollary}

\medskip
This completes the proof of (4) and therefore all of Theorem \ref{theorem}. 
\hfill$\Box$

\subsection{Proof of Proposition \ref{complete_graph}} 
$X$ with $\#X=n$ is a complete graph if and only if the multiplicity of the shortest edge length is $N={n\choose2}$, which can be seen from the coefficient of the first term except for the constant term in $m_X(q)$. 
\hfill$\Box$

\section*{Acknowledgments} 
The author thanks Kiyonori Gomi for helpful suggestions. 
He also thanks the anonymous reviewer for careful reading and for the information of \eqref{eqL}, which simplifies the previous proofs.   

\bibliographystyle{amsplain}


\begin{dajauthors}
\begin{authorinfo}[jo]
  Jun O'Hara\\
Department of Mathematics and Informatics, \\
Faculty of Science, Chiba University\\
1-33 Yayoi-cho, Inage, Chiba, 263-8522, JAPAN  \\
  ohara@math.s.chiba-u.ac.jp \\
  \url{https://sites.google.com/site/junohara/}
\end{authorinfo}
\end{dajauthors}

\end{document}